\documentclass[12pt,a4paper]{amsart}
\usepackage[latin1]{inputenc}
\usepackage[english]{babel}
\usepackage{amstext}
\usepackage{latexsym}
\usepackage{amsmath}
\usepackage{amssymb}
\usepackage{mathtools}
\usepackage{amsfonts}
\usepackage{graphicx}
\usepackage{amssymb}
\usepackage{amsthm}
\usepackage{epsfig}
\usepackage{tikz}
\usetikzlibrary{matrix,arrows,decorations.pathmorphing}
\usepackage{graphicx}
\usepackage{enumitem}
\usepackage{xcolor}

\newtheorem{defn}{Definition}[section]

\newtheorem{lemma}[defn]{Lemma}

\newtheorem{prop}[defn]{Proposition}
\newtheorem{theo}[defn]{Theorem}
\newtheorem{example}[defn]{Example}
\newtheorem{pergunta}[defn]{Question}
\newtheorem{question}[defn]{Question}

\newcommand{\T}{\mathbb{T}^2}

\newcommand{\R}{\mathbb{R}}

\newcommand{\Z}{\mathbb{Z}}
\newcommand{\N}{\mathbb{N}}

\title[Random cocycles in $\T$]
{Rotation sets for random compositions of $\T$ homeomorphisms}

\author{C. Freijo}
\address[Freijo]{Instituto de Matem\' atica e Estat\' istica da Universidade de S\~ao Paulo,
R. do Mat\~ ao, 1010 - Vila Universitaria, S\~ ao Paulo, Brasil}
\email{fabiotal@ime.usp.br}

\author{F.A. Tal}
\address[Tal]{Instituto de Matem\' atica e Estat\' istica da Universidade de S\~ao Paulo,
R. do Mat\~ ao, 1010 - Vila Universitaria, S\~ ao Paulo, Brasil}
\email{fabiotal@ime.usp.br}

\thanks{C. Freijo was supported by Fapesp. F. Tal Was partially supported by Fapesp and CNPq}

\begin{document}

\begin{abstract}

We study cocycles of homeomorphisms of $\T$ in the isotopy class of the identity over shift spaces, using as a tool a novel definition of rotation sets inspired in the classical work of Miziurewicz and Zieman. We discuss different notions of rotation sets, for the full cocyle as well as for measures invariant by the shift dynamics on the base. We present some initial results on the shape of rotation sets, continuity of rotation sets for shift-invariant measures, and bounded displacements for irrotational cocyles, as well as a few  interesting examples in an attempt to motive the development of the topic.
\end{abstract}
	
\maketitle

\section{Introduction}

Rotation sets and rotation theory have become a powerful and extensively studied tool in the understanding of the dynamics of homeomorphisms of the two-dimensional torus. For a torus homeomorphism $g:\T\to\T$ homotopic to the identity, the seminal work of Misiurewicz and Ziemian~\cite{MZ} established that the rotation set $\rho(\widetilde g)$ of a lift $\widetilde g:\R^2\to\R^2$ is always a nonempty compact convex subset of $\R^2$. This structural result laid the foundation for a rich theory connecting geometric properties of rotation sets with dynamical features of the system. In particular, Franks~\cite{Franks} proved that whenever the rotation set has nonempty interior, the system admits infinitely many periodic points with distinct rotation vectors, while Llibre and Mackay~\cite{LM} showed that in this case the topological entropy is necessarily positive.

The purpose of this paper is to explore extensions of this theory to the setting of cocycles whose fiber dynamics is given by torus homeomorphisms in the identity isotopy class. More precisely, we consider cocycles over product spaces, with particular focus on the case where the base is a shift space, and investigate how notions from rotation theory can be adapted to this context. In doing so, we introduce several different definitions of rotation sets for such cocycles and examine how they relate to each other as well as some of their properties. These include their possible ``shapes," their continuity with respect to the base invariant measure, the distinction between essential and inessential points for the dynamics, and the appearance of bounded displacement or sublinear behavior.

For instance, after defining the Misiurewicz-Zieman rotation set of a cocycle, we show that
\begin{theo}
The Misiurewicz-Ziemian type rotation set $\rho_{mz}(\widetilde F)$ associated to a cocycle $\widetilde F$ over the full shift in $\Sigma=\{0,1\}^\mathbb{Z}$ is always a nonempty, compact, and connected subset of $\R^2$.
\end{theo}

Also, we introduce,  in Section~\ref{section:essentialpoints}, the concept of essential  points for the dynamic of the cocycle, trying to emulate the ideas of \cite{inventiones1, fullyessential}, where they were fundamental in describing what constitutes an homeomorphism having a rotational behaviour that ``sees'' the whole torus. We use it to obtain the following connection with the problem of identifying the existence of sublinear displacemnets not captured by the rotation set:

\begin{theo}
For locally constant, area-preserving cocycles $\widetilde F$ with inessential fixed point sets and trivial rotation set, either the set of essential points is an essential subset of $\R^2$ or $\widetilde F$ has uniformly bounded displacements. In particular, the dynamics exhibits a strong bounded deviation property.
\end{theo}

Finally, it is relatively easy to show examples were the Misiurewicz-Zieman rotation set is not convex in this setting. But we can transform the study $\varepsilon$-pseudoorbits of a given homeomorphism map $g$ into a study of cocycles and associate here a rotation set. We show that: 

\begin{theo}
If $g$ is conservative of if the rotation set of a lift $\widetilde g$ of $g$ has nonempty interior, then its $\varepsilon$-pseudoorbits rotation set is convex.
\end{theo}

The paper is intended as a way to open a new direction of study, and it presents several different questions we believe are relevant.

\section{Preliminaries}
Let $X$ be a compact metric space and $\Sigma=X^{\mathbb{Z}}$ be the infinite product space endowed with the product topology, then the left-shift map over $\Sigma$ is the continuous function $\sigma\colon\Sigma\to \Sigma $ satisfying $\sigma\left(({x}_n)_{n\in\mathbb{Z}}\right)=({x}_{n+1})_{n\in\mathbb Z}$. 

For $Y$ a topological space, a {random composition of $Y$ maps over $\sigma$} is defined as the cocycle obtained by the continuous transformation $F\colon \Sigma\times Y\to \Sigma\times Y$, that sends $\{x\}\times Y$ into the iterated fiber $ \{\sigma(x)\} \times Y$ and taking the form 
 $$ F(x,p)=(\sigma(x), f_x(p)),$$
where each $f_x$ is a continuous map of $Y$.  In particular we observe that the iterations of $F$ are given by $$F^n(x,p)=(\sigma^n(x), f^n_x(p)),$$where $f^n_x\coloneq f_{\sigma^{n-1}(x)} f_{\sigma^{n-2}(x)}\ldots f_{\sigma(x)}f_x$.
In this work consider random composition of $Y=\T=\R^2/\Z^2$ homeomorphisms (shortened as RCTH). We denote $M=\Sigma\times \mathbb T^2$, and we study only RCTHs where the action on the fibers $x\mapsto f_x$ takes values in the set of homeomorphisms homotopic to identity $\text{Hom}_0(\mathbb T^2)$, which, by a result of Epstein (\cite{XXXX}), is the same as the set of homeomorphisms of $\T$ isotopic to the identity.

Denote by $\pi:\R^2 \to \T$ to the covering map. A lift of an homeomorphism $f:\T\to\T$ to $\R^2$ is a map $\widetilde f:\R^2\to\R^2$ such that $\pi\circ\widetilde f= f\circ\pi$. It is well known that, if $f$ is a homeomorphisms of $\T$ that is isotopic to the identity,  any lift $\widetilde f$ to $\R^2$ of $f$ comutes with the translations $T_v(\widetilde p)=\widetilde p+v$ whenever $v\in\Z^2$. Furthermore, if $\widetilde f_1$ and $\widetilde f_2$ are two lifts of $f$, then their difference is constant and equal to a vector in $\Z^2$.
Whenever presented with a RCTH $F$, for every $x\in\Sigma$ we fix a lift $\widetilde f_x\colon \mathbb R^2\to \mathbb R^2$ such that the map $x\mapsto \widetilde f_x$ is still continuous. This induces a new cocycle $\widetilde F:\Sigma\times\R^2$ given by
$$\widetilde F(x,\widetilde p)=(\sigma(x), \widetilde f_x(\widetilde p)).$$

\begin{defn}
 We say that a cocycle $F\colon   M\to M$ is locally constant if the map $x\mapsto f_x$ depends only on the zeroth coordinate of $x$, which means that $f_x=f_{y}$ whenever $x_0=y_0$. It is customary, in this case, to abuse notation and denote $f_x$ and $\widetilde f_x$ by $f_{x_0}$ and $\widetilde f_{x_0}$. 
\end{defn}

Finally, a RTCH is said to be \emph{conservative} if for every $x$ in $\Sigma$, the nonwandering set  $\Omega(f_x)=\T$. It is said to be area-preserving if  there exists some borel probability measure positive on open sets and with full support that is invariant by $f_x$ for all $x\in\Sigma$.

\section{Rotation sets}
Given a RTCH and its associated cocyle of plane homeomorphisms, there are several rotational objects that can be used to study its dynamics. To do so, given $F$ and $\widetilde F$, define the displacement function $\rho_1(x,p):M\to\T$ as
$\rho_1(x,p)= \widetilde f_x(x,\widetilde p)-\widetilde p$, where $\widetilde p$ is a point in $\pi^{-1}(p)$. As $\widetilde f_x$ comutes with $\Z^2$ translations, the choice of $\widetilde p$ in $\pi^{-1}(p)$ does not change the value of $\rho_1$. There are two  main ways of understanding rotation for torus homeomorphisms, as time average displacements of orbits and as mean displacements of invariant measures.

\subsection{Time-average rotation vectors and rotation sets}
The time $n$ averaged displacement for a point $(x,p)$ is defined as
$$\rho_{n}(x,p)=\frac{1}{n}\sum_{i=0}^{n-1}\rho_1(\widetilde F^i(x,p))= \frac{\widetilde f^n _x(\widetilde p)-\widetilde p}{n}, \widetilde p\in\pi^{-1}(p).$$
Time averaged rotation vectors are then limits of vectors that are time $n_k$ averaged displacements for some increasing sequence $n_k$.   There are here at least $2$ different notions we want to differentiate, those of rotation sets for points in $M$ and the full Misiurewicz-Zieman rotation set.
 
We will say that a point $(x,p)$ \emph{has a rotation vector} if the limit $\lim_{n\to\infty}\rho_n(x,p)$ exists, however this sequence may not converge even when $\widetilde f_x$ is independent of $x$ (reducing the action on fibers to the classical case of a single homeomorphism). We define then $\rho(x,p, \widetilde F)$ the \emph{rotation set of the point $(x,p)$} as the accumulation points of the sequences of time averaged displacements of $(x,p)$, and we define $\rho_{\text{point}}(\widetilde F)$ the full \emph{point-wise rotation set} as the union of the rotation sets for points in $M$. More precisely:
\begin{defn}
	 $\rho(x,p,\widetilde F)\coloneq\{ \text{accumulation points of } (\rho_n(x, p))_{n\in\N}\},$ and $ \rho_{\text{point}}(\widetilde F)\coloneq\bigcup_{(x,p)\in M}\rho(x,p,\widetilde F).$
\end{defn}  

Second definition is inspired by the classical Misiurewiz-Zieman rotation set:
\begin{defn}
Denote
$$D_n\coloneq\{ \rho_n(x,p)| \, \, (x,p)\in M\},$$
we define the \emph{Misiurewicz-Ziemian rotation set of $\widetilde F$} as
$$\rho_{mz}(\widetilde F)\coloneq \limsup_{n\to\infty} D_n=\bigcap_{N\in\mathbb N}\overline{\bigcup_{n\geq N} D_n}. $$
\end{defn}	
We abuse notation and drop the dependence on $\widetilde F$ for rotation sets whenever the context allows for this. 
	
Of course, $\rho_{point}\subset \rho_{mz}$. There are easy examples where the inclusion is proper when dealing with the classical case of a single torus homeomorphism $g$, but the questions whether this inclusion is proper or not when the rotation set of $g$ has nonempty-interior, or if the non-wandering set of $g$ is the full torus are still open. This opens the following natural question:
\begin{pergunta}
Is there an example of conservative cocycle $F$ such that $\rho_{point}\not=\rho_{mz}$?
\end{pergunta}

\subsection{Rotation sets for measures}

We denote by $\mathcal{M}(\sigma)$ the set of $\sigma$-invariant probability measures in $\Sigma$, and by $\mathcal{M}(F)$ the set of $F$-invariant probability measures in $M$. For every $\mu$ in $\mathcal{M}(\sigma)$, let $\mathcal{M}_{\mu}(F)$ be the subset of $\mathcal{M}(F)$ of measures projecting to $\mu$, that is, the set of me\-a\-su\-res whose push-forward by the projection $P_1(x,p)=x$ is $\mu$. It is a classical result that since $M$ is compact then $\mathcal{M}_{\mu}(F)$ is nonempty, compact and convex for any $\mu$. Finally, let $\mathcal{M}_{\text{Per}}(\sigma)$ be the subset of the measures in $\mathcal{M}(\sigma)$ supported on periodic orbits of $\sigma$.

To follow we give the last notion of rotation sets involved.
\begin{defn}
For any $m$ in $\mathcal{M}(F)$, let 
$$\rho(\widetilde F,m)=\int \rho_1(x,p) dm$$
be the {\emph{rotation vector}} of the measure $m$.

For any $\mu$ in $\mathcal{M}(\sigma)$, we define its rotation set as
$$ \rho_{\mu}(\widetilde F)=\{ \rho(\widetilde F,m)\, | \, m\in \mathcal{M}_{\mu}(F) \},$$
If $\mathcal{M}_{\mu,erg}(F)$ is the subset of ergodic measures of $\mathcal{M}_{\mu}(M)$, we define
$$ \rho_{\mu,erg}(\widetilde F)=\{ \rho(\widetilde F,m) \,| \, m\in \mathcal{M}_{\mu, erg}(F) \}.$$

Finally, we define the full \emph{measure rotation set}, the \emph{ergodic rotation set}  and the \emph{rotation set of periodic words} as, respectively, 
\begin{equation*}\begin{aligned}
\rho_{mes}(\widetilde F) &=&\bigcup_{\mu\in \mathcal{M}(\sigma)}&\{\rho_{\mu}(\widetilde F)\}=&\bigcup_{m\in \mathcal{M}(F)}\{\rho(\widetilde F,m)\},\\
\rho_{erg}(\widetilde F)&=&\bigcup_{\tiny{
	\begin{aligned} m\in \mathcal{M}(F), \\ \mu \text{ is ergodic}\end{aligned}}}&\{\rho(\widetilde F,m)\},\\
\rho_{\text{Per}(\sigma)}(\widetilde F)&=&\bigcup_{\mu\in \mathcal{M}_{\text{Per}}(\sigma)}&\{\rho(\widetilde F,\mu)\}.\end{aligned}
\end{equation*}
\end{defn}
Here the dependence of $\widetilde F$ is dropped whenever the context is clear.

To continue, let us show two results on the connections between the different rotation sets.
\begin{prop}
$$\rho_{\text{point}}(\widetilde F)\subset\rho_{mz}(\widetilde F)\subset  \rho_{mes}(\widetilde F)$$	
\end{prop}

\begin{proof}
While the first inclusion is obvious, the other one comes from the observation that $\rho_n(x,p)=\frac{1}{n}\sum_{i=0}^{n-1}\rho_1(\widetilde F^i(x,p))$. If $v\in\rho_{mz}(\widetilde F)$, then there exists a sequence $(x_k, p_k)_{k\in\N}$ of points in $M$ and a sequence of increasing integers $n_k$ such that $\rho_{n_k}(x_k,p_k)$ converges to $v$. Taking the probability measure $m_k=\frac{1}{n_k}\sum_{i=^0}^{n_k-1}\delta_{F^i(x_k, p_k)}$ which does not need to be invariant, it is possible to observe that $\rho_{n_k}(x_k,p_k)=\int \rho_1 dm_k$. Because of the weak-* compactness of $\mathcal{M}(M)$ there must exists a weak-* limit $m$ of some subsequence $m_{k_l}$. It is a classical Krylov-Boguliov argument that $m$ is $F$-invariant and of course $\rho(\widetilde F,m)=v$.	
\end{proof}
	
\begin{prop}
$\rho_{{}_{\text{Per}(\sigma)}}(\widetilde F)$ is a subset of $\rho_{mz}(\widetilde F).$
\end{prop}
\begin{proof}
	
	Let $x\in \Sigma$ and $n$ be a positive integer such that $\sigma^n(x)=x$. Define $g=f^n_x$ and let $\widetilde g=\widetilde f^n_x$ be its lift to $\mathbb{R}^2$.
	
	First we prove that $\rho(\widetilde g)/n \subset \rho_{mz}(\widetilde F)$. Let $v\in\rho(\widetilde g)$, there exists a sequence of points $(\widetilde p_k)_{k\in\N}\subset \R^2$ and a sequence of integers $j_k$ such that
	$$\lim_{k\to\infty}\frac{1}{j_k}\left(\widetilde g^{j_k}(\widetilde p_k)-\widetilde p_k\right)=v.$$
	Observe that $\widetilde g^{j_k} = \widetilde f_x^{n j_k}$. Therefore, this implies
	$$\lim_{k\to\infty}\frac{1}{n j_k}\left(\widetilde f_x^{n j_k}(\widetilde p_k)-\widetilde p_k\right)=\frac{v}{n},$$
	which shows that $v/n \in \rho_{mz}(\widetilde F)$, and hence $v \in n\rho_{mz}(\widetilde F)$.
	
	If $\mathcal{M}(g)$ is the set of all $g$-invariant Borel probability measures, then
	$$\rho(\widetilde g)=\left\{\int \rho_n(x,p)  d\mu(p) \ : \ \mu\in \mathcal{M}(g)\right\},$$
	then it remains to show that for any $F$-invariant measure $m$ whose projection onto $\Sigma$ is supported on the orbit of $x$, there exists a corresponding $\mu$ in $\mathcal{M}(g)$ such that
	$$\int \rho_1(y,p)  dm(y,p)=\frac{1}{n}\int \rho_n(x,p)  d\mu(p),$$
	proving that $\rho_{Per(\sigma)}(\widetilde{F})=\rho(\widetilde g)/n$. Indeed, for any $\mu\in\mathcal M(g)$, the measure $m$ defined by
	$$m=\frac{1}{n}\sum_{i=0}^{n-1}\delta_{\sigma^{i}(x)}\times (f^{-i}_x)_*(\mu)$$
	projects to the measure supported on the orbit of $x$ and is $F$-invariant. Moreover, we compute
	\begin{align*}
		\int \rho_1(y,p)  dm(y,p) &=\frac{1}{n}\sum_{i=0}^{n-1}\int \rho_1(y,p)  d\left(\delta_{\sigma^{i}(x)}\times (f^{-i}_x)_*(\mu)\right)(y,p) \\
		&=\frac{1}{n}\sum_{i=0}^{n-1}\int_{\mathbb{T}^2} \rho_1(\sigma^{i}(x), p)  d\left((f^{-i}_x)_*(\mu)\right)(p) \\
		&=\frac{1}{n}\sum_{i=0}^{n-1}\int_{\mathbb{T}^2} \rho_1(\sigma^{i}(x), f^{i}_x(q))  d\mu(q) \\
		&=\frac{1}{n}\int_{\mathbb{T}^2} \left( \sum_{i=0}^{n-1} \rho_1(\sigma^{i}(x), f^{i}_x(q)) \right) d\mu(q) \\
		&=\frac{1}{n}\int_{\mathbb{T}^2} \rho_n(x,q)  d\mu(q).
	\end{align*}
	This completes the verification.
	
\end{proof}

Here we pose two different questions, one in the spirit of the traditional study of rotation theory, and another interesting one mimicking other in the study of cocycles is how well do periodic words capture the rotational behavior.
\begin{question}
Given a RCTH, is $\rho_{mz}=\rho_{mes}$?
\end{question}
\begin{question}
Given a locally constant RCTH, is $\overline{\rho_{{}_{\text{Per}(\sigma)}}}$ equal to $\rho_{mz}$?
\end{question}

\section{Continuity properties}

A natural question, which mirrors a well-studied problem concerning Lyapunov exponents, is whether the rotation sets for shift-invariant probability measures on $\Sigma$ vary continuously. The answer, of course, may depend on the chosen notion of convergence for the space of measures. As in the case of Lyapunov exponents, continuity is not guaranteed if one considers only weak-* convergence, as the following example demonstrates:

\begin{example}
Let $F$ be a locally constant RCTH defined on $\{0,1\}^{\Z}\times\T$, where $\widetilde f_0(p_1,p_2)=(p_1+ sin(2\pi p_2), p_2)$ and $\widetilde f_1(p_1,p_2)=(p_1, p_2+\sqrt{2})$. Let $\mu_0$ be the measure supported on the null sequence $x^0$ where $(x^0)_i=0, i\in\Z$ and let $\mu_k$ is the invariant-measure supported on the periodic orbit of the $k$-periodic point $x^k$, where $({x^k})_{i}=0, 0\le i<k-1$ and $({x^k})_{k-1}=1$. Then $\rho_{\mu_0}(\widetilde F)=[-1,1]\times\{0\}$ while $\rho_{\mu_k}(\widetilde F)=\{(0,\sqrt{2}/k)\}$.
\end{example}
\begin{proof}
	
	We only sketch the main argument here. Since the standard definition of the Misiurewicz-Ziemian rotation set for a homeomorphism coincides with the measure rotation set, we have $\rho_{\mu_0}(\widetilde F) = \rho(\widetilde{f}_0)$. This follows because any measure $m \in \mathcal{M}_{\mu_0}( F)$ can be decomposed as a product measure $\delta_{x^0} \times \nu$, where $\nu$ is a measure on $\mathbb{T}^2$ invariant under $f_0$.
	
	On the other hand, if $(y, q)$ is a point in the support of an ergodic measure $m \in \mathcal{M}_{\mu_k, \text{erg}}( F)$, then there exists $0 \le i < k$ such that $\sigma^i(y) = x^k$. By Birkhoff's Ergodic Theorem, $$\rho_m(\widetilde F) = \lim_{n \to \infty} \rho_n(y, q) = \lim_{n \to \infty} \rho_n(\widetilde F^i(y, q)),$$thus we may assume without loss of generality that $y = x^k$.
	
	Taking $n = lk$, we observe that
	$$
	\rho_n(x^k, q) = \frac{1}{n} \sum_{j=0}^{l-1} \rho_k(\widetilde F^j(x^k, q)) = \frac{1}{n}\left((k-1) \sum_{j=0}^{l-1}  \sin(2\pi (q_2 + j\sqrt{2})), l\sqrt{2} \right).
	$$
	Since
	$$
	\lim_{l\to\infty}\frac{1}{l} \sum_{j=0}^{l-1} \sin(2\pi (q_2 + j\sqrt{2})) = \int_0^1 \sin(2\pi t)  dt = 0,
	$$
	the result follows.
	
\end{proof}

On the other hand, while one cannot expect continuity, at least semi-continuity holds

\begin{prop}
Let $(\mu_k)_{k\in\N}$ be a sequence in $\mathcal{M}(\sigma)$ converging in the weak-* topology to a measure $\mu$, and let $(v_k)_{k\in\N}$ be a sequence with $v_k\in\rho_{\mu_k}(\widetilde F)$ converging to $v$, then $v\in\rho_{\mu}(\widetilde F)$.
\end{prop}
\begin{proof}
Since $v_k \in \rho_{\mu_k}(\widetilde F)$, for each $k$ there exists an $F$-invariant probability measure $m_k$ projecting to $\mu_k$ such that $\rho_{m_k}(\widetilde F) = v_k$. By the compactness of the space of probability measures in the weak-$^*$ topology, there exists a subsequence $(m_{k_l})_{l \in \mathbb{N}}$ converging to some measure $m$.

The weak-$*$ convergence of $m_{k_l}$ to $m$, together with the fact that each $m_{k_l}$ projects to $\mu_{k_l}$ and $\mu_{k_l}$ converges to $\mu$, implies that $m$ projects to $\mu$. Moreover, since the rotation vector depends continuously on the measure in this topology, we have:
\[
\rho_m(\widetilde F) = \int \rho_1  dm = \lim_{l \to \infty} \int \rho_1  dm_{k_l} = \lim_{l \to \infty} v_{k_l} = v.
\]
Hence, $v \in \rho_\mu(\widetilde F)$.
\end{proof}

The preceding discussion naturally leads to the following question:

\begin{question}
	Suppose $(\mu_k)_{k \in \mathbb{N}}$ is a sequence in $\mathcal{M}(\sigma)$ converging to $\mu$ in the weak-$*$ topology. If, in addition, the supports of the measures $\mu_k$ converge to the support of $\mu$ in the Hausdorff topology, does it follow that the associated rotation sets converge in the Hausdorff topology as well? That is, does
	\[
	\mu_k \overset{\text{weak-}*}{\longrightarrow} \mu \quad \text{and} \quad \operatorname{supp}(\mu_k) \overset{H}{\longrightarrow} \operatorname{supp}(\mu)
	\]
	imply
	\[
	\rho_{\mu_k}(\widetilde F) \overset{H}{\longrightarrow} \rho_{\mu}(\widetilde F) \quad ?
	\]
\end{question}

\section{Essential and Inessential points}\label{section:essentialpoints}

In this section, we assume that $\Sigma = \{0,1\}^{\mathbb{Z}}$ and that the homeomorphism $f_x$ takes values in the set $\{f_0, f_1\}$, where $f_0$ and $f_1$ are area-preserving homeomorphisms of the torus $\mathbb{T}^2$, each homotopic to the identity.

\begin{defn}
	An open subset $U \subset \mathbb{T}^2$ is called inessential if every loop contained in $U$ is homotopically trivial in $\mathbb{T}^2$; otherwise, $U$ is called essential. A general set $E \subset \mathbb{T}^2$ is inessential if it is contained in some inessential open set. A set $E$ is fully essential if its complement $\mathbb{T}^2 \setminus E$ is inessential.
	
	A point $p \in \mathbb{T}^2$ is called inessential for the cocycle $F$ if there exists a neighborhood $B$ of $p$ such that the set
	\begin{equation}\label{up}
U_B = \bigcup_{x \in \Sigma} \bigcup_{n \in \mathbb{N}} f_x^n(B)
	\end{equation}
	is inessential. If $p$ is not inessential for $F$, it is said to be essential for $F$.
\end{defn}

Here we denote by $\text{Ess}(F)$ and $\text{Ine}(F)$ the sets of essential and inessential points of $F$ respectively. 

\begin{lemma}
	The set $\text{Ine}(F)$ is $f_i$-invariant for every $i\in\{0,1\}$.
\end{lemma}
\begin{proof}
Note that for every $p \in \text{Ine}(F)$, if $B$ is a neighborhood of $p$ such that $U_B$ is inessential, then for each $i \in \{0,1\}$, the set $f_i(B)$ is a neighborhood of $f_i(p)$, and
\[
U_{f_i(B)} = f_i(U_B)
\]
is also inessential. Consequently, $p$ is inessential for $F$ if and only if $f_i(p)$ is inessential for $F$.
\end{proof}

\begin{defn}
If $O$ is an open subset of $\mathbb{T}^2$, then its filling, denoted $\text{Fill}(O)$, is defined as the union of $O$ with all inessential connected components of its complement.

The regularized filling of $O$, denoted $\text{Fill}^\prime(O)$, is the set $\pi(A)$, where $A$ is the union of the interiors of the closures of all connected components $\widetilde{W}$ of $\pi^{-1}(\text{Fill}(O))$. Equivalently, $A$ consists of all points $\widetilde{p} \in \mathbb{R}^2$ for which there exists a connected component of $\pi^{-1}(\text{Fill}(O))$ that is dense neighborhood of $\widetilde{p}$.
\end{defn}

The following properties are direct consequences of the definitions. If $O$ is an inessential set, then both its filling $\text{Fill}(O)$ and its regularized filling $\text{Fill}'(O)$ are also inessential. Furthermore, if a homeomorphism $g$ leaves $O$ forward invariant, then both $\text{Fill}(O)$ and $\text{Fill}'(O)$ are forward invariant under $g$ as well. The next lemma establishes a stronger invariance property in the area-preserving case.

\begin{lemma}\label{lemma:primefilled}
	If $g$ is an area-preserving homeomorphism and $O$ is a forward invariant inessential open set, then the regularized filling $\text{Fill}'(O)$ is invariant under $g$.
\end{lemma}
\begin{proof}

We already know that $\text{Fill}'(O)$ is forward invariant. To prove full invariance, first observe that the area of any connected component of $\text{Fill}'(O)$ must belong to a countable set of positive values. Let $\{a_1, a_2, a_3, \dots\}$ be a strictly decreasing sequence of these possible areas, converging to zero.

Now, consider a connected component $W$ of $\text{Fill}'(O)$ with area $a_1$. Since $g$ is area-preserving, the image $g(W)$ is contained in some connected component $W'$ of $\text{Fill}'(O)$ with the same area, $a_1$. Moreover, since $\text{Fill}'(O)$ is forward invariant and $g$ is a homeomorphism, $g$ maps $W$ densely into $W'$. Because there are only finitely many components of area $a_1$, $g$ must permute these components.

A similar argument, applied inductively to the components of area $a_j$ for each $j$, shows that $g$ permutes all connected components of $\text{Fill}'(O)$. Therefore, $\text{Fill}'(O)$ is invariant under $g$.
\end{proof}

If $U$ is an open topological disk in $\mathbb T^2$, $\mathcal{D}(U)$ is the diameter of any connected component of $\widetilde U$. 

We recall here the main result of \cite{tripleboundary} (see also \cite{inventiones1})
\begin{theo}\label{tripleboundary}
If $g$ is an area preserving homeomorphism of $\T$ isotopic to the identity such that $\text{Fix}(g)$ is inessential then there exists $M>0$ such that if  $U$ is an open topological disk and $g(U)=U$, then $\mathcal{D}(U)<M$.
\end{theo}

\begin{lemma}\label{lema9}
Assume that $\rho_{mz}(\widetilde F) = \{0\}$ and that the set $\text{Fix}(f_0^j)$ is inessential for every $j \in \mathbb{N}$. Then, for every $p \in \text{Ine}(F)$, there exist a constant $M_p > 0$ and an open set $D_p$ containing $p$ such that  
\[
\left\| \widetilde{f}_x^n(\widetilde{q}) - \widetilde{q} \right\| < M_p \quad \text{for all } (x, q) \in \Sigma \times D_p \text{ and all } n \in \mathbb{N}.
\]
	
\end{lemma}

\begin{proof}
Let $p \in \text{Ine}(F)$. Since $p$ is an inessential point for $F$, there exists an open neighborhood $B$ of $p$ such that the set $U_p$ defined in (\ref{up}) is inessential. In particular, $U_p$ is forward invariant under both $f_0$ and $f_1$. Consider its regularized filling $U'_p = \text{Fill}'(U_p)$, which, by Lemma~\ref{lemma:primefilled}, is a union of disjoint topological disks invariant under both $f_0$ and $f_1$.

Let $D_1$ denote the connected component of $U'_p$ containing $p$. Since the total area is finite, there are only finitely many components; denote them by $D_1, \ldots, D_{r_p}$. These form a finite collection of disjoint disks that are permuted by both $f_0$ and $f_1$. Furthermore, since $f_0$ is nonwandering, there exists an integer $j$ such that each disk is invariant under $g = f_0^j$. By Theorem~\ref{tripleboundary}, there exists a constant $M_0 > 0$ such that the diameter of each $D_k$ satisfies $\mathcal{D}(D_k) < M_0$ for all $1 \le k \le r_p$.

	Fix, for each \(1 \le k \le r_p\), a lift \(\widetilde{D}_k\) of the disk \(D_k\), and choose a lift \(\widetilde{p}\) of \(p\) contained in \(\widetilde{D}_1\). We now show that for each \(k\), there exists a unique vector \(v_k \in \mathbb{Z}^2\) such that for any \(n \in \mathbb{N}\) and any \(x \in \Sigma\) for which the cocycle iteration satisfies \(f_x^n(D_1) = D_k\), the corresponding lift satisfies
	\[
	\widetilde{f}_x^n(\widetilde{D}_1) = \widetilde{D}_k + v_k.
	\]
	
Indeed, suppose for a contradiction that there exists an index \(1 \le k \le r_p\) and two different words \(x_1, x_2 \in \Sigma\) with corresponding iterates \(n_1, n_2 \in \mathbb{N}\) such that \(f_{x_1}^{n_1}(D_1) = D_k\) and \(f_{x_2}^{n_2}(D_1) = D_k\), but with two distinct integer vectors \(v_1 \neq v_2 \in \mathbb{Z}^2\) satisfying
\[
\widetilde{f}_{x_1}^{n_1}(\widetilde{D}_1) = \widetilde{D}_k + v_1 \quad \text{and} \quad \widetilde{f}_{x_2}^{n_2}(\widetilde{D}_1) = \widetilde{D}_k + v_2.
\]
Without loss of generality, we may assume \(v_1 = 0\) (by choosing an appropriate lift of \(D_k\)). Since \(f_{x_1}^{n_1}\) is nonwandering, there exists an integer \(\alpha > 0\) such that \(\left(f_{x_1}^{n_1}\right)^\alpha(D_1) = D_1\). We now construct a periodic element \(z \in \Sigma\) with period \(t = (\alpha - 1)n_1 + n_2\) such that the composition of maps takes points in \(\widetilde{D}_k\) to points in \(\widetilde{D}_k + v_2\), which will generate a nontrivial rotation vector, contradicting the assumption that \(\rho_{mz}(\widetilde F) = \{0\}\).

We now construct the periodic element \(z \in \Sigma\) as follows. Define the first \(t = (\alpha - 1)n_1 + n_2\) symbols of \(z\) by
$$z_i=\left\lbrace\begin{aligned}
	&x_{1,i-[i/n_1]n_1} &\text{ if }&  i=1\ldots (\alpha-1) n_1 ; \\
	&x_{2,i- (\alpha-1) n_1}&\text{ if } & i=(\alpha-1) n_1+1\ldots (\alpha-1) n_1+n_2, \\
\end{aligned}\right.$$
and extend \(z\) periodically with period \(t\). This construction ensures that the composition of the corresponding maps generates rotation for the elements in $\widetilde{D}_k$, thus 
\begin{eqnarray}\label{eq8}
	\widetilde f_z^{(\alpha-1) n_1+n_2}(\widetilde D_k)=&\widetilde f_z^{(\alpha-1) n_1+n_2}(\widetilde f^{n_1}_{x_1}(\widetilde D_1))\\ \nonumber
	=&\widetilde f^{n_2}_{x_2} \circ (\widetilde f_{x_1}^{n_1}(\widetilde D_1))^{\alpha}\\
	=&\widetilde f^{n_2}_{x_2}(\widetilde D_1)=\widetilde D_k +v_2. \nonumber
\end{eqnarray}
For any \(j \in \mathbb{N}\), let \(n_j = j t\). Then, for any point \(\widetilde{q} \in \widetilde{D}_k\), its iterate $
\widetilde{f}_z^{n_j}(\widetilde{q}) \in \widetilde{D}_k + j v_2.$
This implies that \(\widetilde{f}_z^{n_j}(\widetilde{q}) - j v_2 \in \widetilde{D}_k\), and since \(\mathcal{D}(\widetilde{D}_k) < M_0\), we have
\[
\left\| \widetilde{f}_z^{n_j}(\widetilde{q}) - j v_2 - \widetilde{q} \right\| < M_0.
\]

We now compute the rotation vector for the point \((z, \widetilde q)\). Since
\[
\frac{\widetilde{f}_z^{n_j}(\widetilde{q}) - \widetilde{q}}{n_j} = \frac{\widetilde{f}_z^{n_j}(\widetilde{q}) - j v_2 - \widetilde{q}}{j t} + \frac{v_2}{t},
\]
together with the fact that the first term on the right-hand side converges to zero, we have that
\[
\lim_{j \to \infty} \frac{\widetilde{f}_z^{n_j}(\widetilde{q}) - \widetilde{q}}{n_j} = \frac{v_2}{t} \neq 0,
\]
which contradicts the assumption that \(\rho_{mz}(\widetilde F) = \{0\}\).

This contradiction shows that the displacement vector \(v_k\) must be unique for each disk \(D_k\). Consequently, the orbit \(\widetilde{f}_x^n(\widetilde{p})\) is confined to the bounded set \(\bigcup_{k=1}^r (\widetilde{D}_k + v_k)\), which completes the proof of the lemma.

\end{proof}

\begin{prop}
Suppose that $\text{Fix}(f^j_0)$ is inessential for every $j\in\mathbb{N}$, and $\rho_{mz}(\widetilde F)$ is the null vector. If $\text{Ine}(F)$ is fully essential for $F$, then there exists a constant $M > 0$ such that
\[
\left\| \widetilde{f}_x^n(\widetilde{p}) - \widetilde{p} \right\| < M \quad \text{for all } (x, p) \in \Sigma \times \mathbb{T}^2 \text{ and } n \in \mathbb{N}.
\]
\end{prop}

\begin{proof}
	
	Since $\text{Ine}(F)$ is fully essential, it contains two non-homotopically trivial loops $\gamma_1$ and $\gamma_2$. Consider their lifts $\widetilde{\gamma}_1$ and $\widetilde{\gamma}_2$ in $\mathbb{R}^2$. There exist sufficiently large integer vectors $v_1, v_2 \in \mathbb{Z}^2$ such that there exists a connected component of the complement of these curves is bounded and contains contains a fundamental domain of the torus. Let $\Gamma$ be the boundary of this connected component.

 For every point $p$ on $\Gamma$, let $D_p$, $M_p$ be the open set and the constant given by Lemma~\ref{lema9}, and let for every given $\widetilde{p}$, let $\widetilde{D}_p$ be the connected component of $\pi^{-1}(D_p)$ that contains a lift $\widetilde{p}$ of $p$. Since $\Gamma$ is compact, it can be covered by finitely many such disks $\widetilde{D}_{p_1}, \ldots, \widetilde{D}_{p_n}$. By Lemma~\ref{lema9}, for each $j$ there exists $M_j > 0$ such that
	\[
	\left\| \widetilde{f}_x^n(\widetilde{q}) - \widetilde{q} \right\| < M_j \quad \text{for all } n \in \mathbb{N} \text{ and all } (x, \widetilde{q}) \in \Sigma \times \widetilde{D}_{p_j}.
	\]
	Let \( M' = \max_j \{ M_j + \text{diam}(\widetilde{D}_{p_j}) \} \). Then, the image of $\Gamma$ under any iterate of the cocycle remains uniformly bounded, this means
	\[
	\left\| \widetilde{f}_x^n(\Gamma) \right\| < M' \quad \text{for all } n \in \mathbb{N} \text{ and all } x \in \Sigma.
	\]
	
	Now, let $U$ be the bounded connected component of the complement of $\Gamma$. By construction, $U$ contains a fundamental domain $D$. Since $\partial \widetilde{f}_x^n(U) = \widetilde{f}_x^n(\Gamma)$ is uniformly bounded and $\widetilde{f}_x^n(U)$ is connected, it follows that $\widetilde{f}_x^n(U)$ is also uniformly bounded for all $x$ and $n$. Therefore, for any $(x, \widetilde{p}) \in \Sigma \times D$ and any $n \in \mathbb{N}$, we have
	\[
	\left\| \widetilde{f}_x^n(\widetilde{p}) - \widetilde{p} \right\| \leq \left\| \widetilde{f}_x^n(\widetilde{p}) \right\| + \left\| \widetilde{p} \right\| < M' + \text{diam}(D),
	\]
	which completes the proof.
\end{proof}

\section{Shape of rotation sets}
 
 When considering a single homeomorphism $g \in \text{Homeo}_0(\mathbb{T}^2)$, it is well known that the Misiurewicz-Ziemian rotation set $\rho(\widetilde{g})$ is always a compact convex subset of $\mathbb{R}^2$~\cite{MZ}. However, the question of which compact convex sets can arise as rotation sets remains largely open. While there are examples of compact convex sets that cannot be rotation sets, it is still unknown, for instance, whether the unit disk can be the rotation set of a torus homeomorphism, or whether this is possible for any convex set with unaccountably many extremal points.
 
 In the context of random cocycles over topological shifts (RCTH), this property does not generalize: rotation sets need not be convex, even in the locally constant case over a finite shift.

\begin{example}
Let $\phi$ be an homeomorphism in $\mathbb{T}^2$ homotopic to the identity, having a lift $\widetilde \phi$ with rotation set $\rho(\widetilde \phi)= [0,1]^2$. 

If $X=\{ 0,1\}$, we define the action of the cocycle $\widetilde{F}$ in the fibers by $\widetilde{f}_0=\widetilde{\phi}$ and $\widetilde{f}_1=\widetilde{\phi}^{-1}$. Clearly, since $\rho(\widetilde \phi^{-1})=-\rho(\widetilde \phi)$, we get that the rotation set of $F$ contains $[-1,0]^2\cup[0,1]^2$, but since $\widetilde{f}^n_x$ is always equal to $\widetilde{\phi}^{j}$ for some $-n\le j\le n$, one sees that $\rho_{mz}(\widetilde{F})=[-1,0]^2\cup[0,1]^2$,  which is not convex.
\end{example}

However we can assure that the Miziurewicz-Zieman rotation set is always connected, even for the non-locally constant case.

\begin{prop}
$\rho_{mz}(\widetilde F)$ is always connected.
\end{prop}
\begin{proof}
We already know that $\rho_{mz}$ is compact. Assume, for contradiction, that it is not connected. Then there exist two disjoint closed sets $K_1$ and $K_2$ such that $\rho_{mz} \subset K_1 \cup K_2$, with $\text{dist}(K_1, K_2) \geq 3\delta$ for some $\delta>0$, and there exist rotation vectors $v_1 \in K_1 \cap \rho_{mz}$ and $v_2 \in K_2 \cap \rho_{mz}$.

By the definition of rotation set, there must exist some $(x, p) \in \Sigma \times \mathbb{T}^2$ and $n_0 \in \mathbb{N}$ such that for all $n > n_0$ the average rotation vector $\rho_n(x, p)$ lies within a $\delta$-neighborhood of $K_1 \cup K_2$.

Furthermore, since the increments between successive averages are small, there exists $n_1 > n_0$ such that for all $(x, p)$ and all $n > n_1$,
\[
\| \rho_{n+1}(x, p) - \rho_n(x, p) \| < \delta.
\]
This implies that, if for some $n > n_1$ we have $\rho_n(x, p)$ belongs to the $\delta$- neighborhood of $K_1$, denoted by $W_1$, then for all $n_2 >n$, $\rho_{n_2}(x, p) \in W_1$ as well.

Now, for a fixed $x$, the map $q \mapsto \rho_{n_2}(x, q)$ is continuous, thus $ \rho_{n_2}(x,\mathbb{T}^2)$ is connected. Therefore, if $\rho_{n_2}(x, p) \in W_1$ for some $p$, then the entire image $\rho_{n_2}(x, \mathbb{T}^2)$ must be contained in $W_1$.

Finally, by the absolute continuity of $\widetilde{F}$, there exists an integer $L > n_2$ such that if two sequences $y$ and $y'$ in $\Sigma$ agree on all coordinates $|j| < L$, then for every $p \in \mathbb{T}^2$ the set $\rho_{n_1}(y, p) \in W_1$ if and only if $ \rho_{n_1}(y', p) \in W_1.$

Now, take a sequence $x$ such that $\rho_{n_1}(x, p) \in W_1$ for some $p$, and another sequence $x'$ such that $\rho_{n_1}(x', q) \in W_2$, implying that $\rho_{n_1}(x', \T) \subset W_2$. We construct a new sequence $y$ as
\[
y_j = 
\begin{cases}
	x_j & \text{for } j \leq L, \\
	x'_{j-2L} & \text{for } j > L.
\end{cases}
\]
By the absolute continuity property, since $y$ agrees with $x$ on the first $L$ coordinates, we have $\rho_{n_1}(y, p) \in W_1$ for all $p \in \mathbb{T}^2$.

Now, consider the average rotation vector at time $2L + M$ for some $M>0$ by
\[
\rho_{2L + M}(y, p) = \frac{2L \rho_{2L}(y, p) + M \rho_M(\sigma^{2L}(y), \widetilde{f}^{2L}_y(p))}{2L + M}.
\]
As $M$ increases, the term ${2L}/({2L + M})$ approaches to zero, thus
\[
\lim_{M \to \infty} \left\| \rho_{2L + M}(y, p) - \rho_M(\sigma^{2L}(y), \widetilde{f}^{2L}_y(p)) \right\| = 0.
\]
Note that for $M > L$, the sequence $\sigma^{2L}(y)$ coincides with $x'$ in the first $L$ coordinates. Therefore, by the absolute continuity property again, $\rho_M(\sigma^{2L}(y), \widetilde{f}^{2L}_y(p)) \in W_2$ for all $p$. Consequently, for sufficiently large $M$, $\rho_{2L + M}(y, p)$ must also lie in $W_2$.

However, since $y$ agrees with $x$ on the first $L$ coordinates, we also have $\rho_{2L + M}(y, p) \in W_1$ for all $M$. This is a contradiction, as $W_1$ and $W_2$ are disjoint. Therefore, the initial assumption that $\rho_{mz}$ is disconnected must be false.

\end{proof}

Furthermore, as mentioned in the introduction, it is shown in \cite{catainterior} that for the locally constant case, a dichotomy analogous to the classical one holds: either the rotation set $\rho_{\text{Per}(\sigma)}$ is contained in a line segment, or it must have nonempty interior. In the remainder of this section, we will demonstrate a specific scenario in which rotation sets are known to be convex, for $\epsilon$-pseudo-orbits of a single torus homeomorphism. We will consider two cases: when the map is conservative, and when the classical rotation set has nonempty interior.
	
\subsection{Case $\epsilon$-pseudo orbits of conservative homeos}	
	
Our aim in this section is to define a cocycle $F$ that encodes the dynamics of $\epsilon$-pseudo-orbits for a given small $\epsilon > 0$.

To construct this, let $X = B_{\epsilon}(0) \subset \mathbb{T}^2$ be the open ball of radius $\epsilon$ centered at $0$. Here, the shift space $\Sigma = X^{\mathbb{Z}}$ represents all possible sequences of perturbations. We fix a conservative homeomorphism $g \in \text{Hom}_0(\mathbb{T}^2)$ homotopic to the identity, along with its lift $\widetilde{g}$ to $\mathbb{R}^2$, and define the cocycle $F: \Sigma \times \mathbb{T}^2 \to \Sigma \times \mathbb{T}^2$ by
\[
F(x, p) = (\sigma(x), g(p) + x_0),
\]
with a lift $\widetilde{F}: \Sigma \times \mathbb{R}^2 \to \Sigma \times \mathbb{R}^2$ chosen to be continuous and such that
\[
\widetilde{F}(x, \widetilde{p}) = (\sigma(x), \widetilde{g}(\widetilde{p})) \quad \text{if } x_0 = 0.
\]
It is clear that the orbit of a point $(x, p)$ under the fiber maps $f_x(p) = g(p) + x_0$ corresponds to an $\epsilon$-pseudo-orbit of $g$ in $\mathbb{T}^2$ starting at $p$.

	We now define the sets of points that can be connected to a point \( p \) via an \(\epsilon\)-pseudo-orbit of a given length.
	
\begin{defn}
 The set of points from which there exists an \(\epsilon\)-pseudo-orbit reaching \( p \) in exactly \( k \) steps is defined by
	\[
	\Theta^{-,k}_{\epsilon}(p) = \left\{ q \in \mathbb{T}^2 \ \middle| \ \exists\, x \in \Sigma \text{ such that } f_x^k(q) = p \right\}.
	\]
	The union over all \( k \geq 0 \) gives the set of all points that eventually reach \( p \)
	\[
	\Theta^{-}_{\epsilon}(p) = \bigcup_{k \geq 0} \Theta^{-,k}_{\epsilon}(p).
	\]
	Analogously, we define the set of points that can be reached from \( p \) in exactly \( k \) steps by
	\[
	\Theta^{+,k}_{\epsilon}(p) = \left\{ q \in \mathbb{T}^2 \ \middle| \ \exists\, x \in \Sigma \text{ such that } f_x^k(p) = q \right\},
	\]
	and the union over all \( k \geq 0 \) gives the set of all points that are eventually reachable from \( p \)
	\[
	\Theta^{+}_{\epsilon}(p) = \bigcup_{k \geq 0} \Theta^{+,k}_{\epsilon}(p).
	\]

\end{defn}

\begin{lemma}\label{1}
For $g$ a conservative map, there exists $N_0>0$ such that for any pair $p,q\in\mathbb T^2$ we can find a $x=x(p,q)$ and $N<N_0$ satisfying $f^N_{x}(p)=q$. 	
\end{lemma}

\begin{proof}
	For a fixed point \( p \in \mathbb{T}^2 \), we show that the set \(\Theta^{+}_{\epsilon}(p)\) covers the entire torus \( \mathbb{T}^2 \).
	
	Suppose, for contradiction, that there exists a point \( y \in \partial \Theta^{+}_{\epsilon}(p) \). Then there exists \( z \in \Theta^{+}_{\epsilon}(p) \) such that \( \|y - z\| < \epsilon/4 \). Since \( z \in \Theta^{+}_{\epsilon}(p) \) and it is non-wandering for \( g \), there exists \( w \in B_{\epsilon/4}(z) \) and an integer \( N > 0 \) such that \( g^N(w) \in B_{\epsilon/4}(z) \). By construction, both \( w \) and \( g^N(w) \) belong to \( \Theta^{+}_{\epsilon}(p) \), and hence the entire ball \( B_\epsilon(g^N(w)) \) is contained in \( \Theta^{+}_{\epsilon}(p) \). Since \( y \in B_\epsilon(g^N(w)) \), it follows that \( y \in \Theta^{+}_{\epsilon}(p) \), contradicting the assumption that \( y \) is a boundary point. Therefore, \( \Theta^{+}_{\epsilon}(p) = \mathbb{T}^2 \).
	
	Next, we show that there is a uniform bound on the length of pseudo-orbits connecting any two points. Since each set $\Theta^{+,N}_{\epsilon}(p)$ is open and their union over \( N \) covers \( \mathbb{T}^2 \), there exists a finite collection \( N_1, \ldots, N_m \) such that
	\[
	\Theta^{+}_{\epsilon}(p) = \bigcup_{i=1}^m \Theta^{+,N_i}_{\epsilon}(p).
	\]
	This yields a maximum length \( N_m \) for pseudo-orbits starting at this specific \( p \). 
	
	Let \( \delta > 0 \) be a constant such that \( B_{\delta}(p) \subset g^{-1}(B_{\epsilon}(g(p))) \). For any pair \( (p, q) \in \mathbb{T}^2 \times \mathbb{T}^2 \), we know there exists a maximum length \( N_0 \in \mathbb{N} \) such that \( q \in \Theta^{+,N_0}_{\epsilon}(p) \). This construction allows us to form an \(\epsilon\)-pseudo-orbit from any point in \( B_{\delta}(p) \) to any point in \( B_{\epsilon}(q) \) in \( N_0 \) steps, implying that the maximum length \( N_0 \) is constant on \( B_{\delta}(p) \times B_{\epsilon}(q) \). If we construct an open cover of $\mathbb T^2\times \mathbb T^2$ with this logic, by compactness we  can subtract a finite sub cover of it. Since the open sets overlap and the maximum length remains constant in each one of them, then we conclude that it is in fact constant for any pair $(p,q)\in \mathbb{T}^2\times\mathbb{T}^2$.

\end{proof}
	
\begin{prop}
	If $g$ is conservative then $\rho_{mz}(\widetilde F)$ is convex. 
\end{prop}	

\begin{proof}
	 
	 We prove that given two points $v_1,v_2 \in \rho_{mz}(\widetilde F)$ any convex combination of them is contained in $\rho_{mz}(\widetilde F)$. 

First we fix $N_0>0$ as in Lemma \ref{1}.
 
Let $(x_{n_k},p_{n_k})$ and $(y_{n_k},q_{n_k} )$ sub-sequences such that
$$\label{limit}\lim_{k\to\infty} \frac{\widetilde{f}^{n_k}_{{x}_{n_k}}(\widetilde p_{n_k})-\widetilde{p}_{n_k} }{{n_k}}=v_1\text{ and }\lim_{k\to\infty}\frac{\widetilde{f}^{m_l}_{{y}_{m_l}}(\widetilde q_{m_l})-\widetilde q_{m_l}}{m_l}=v_2, $$
this means that it is possible to chose for $\delta>0$ integers $k_1,l_1>0$ that for $k>k_1,$ and $l>l_1$
$$\left\| \frac{\widetilde{f}^{n_k}_{{x}_{n_k}}(\widetilde p_{n_k})-\widetilde{p}_{n_k} }{{n_k}}-v_1\right\|<\frac{\delta}{6} \text{ and }\left\|\frac{\widetilde{f}^{m_l}_{{y}_{m_l}}(\widetilde q_{m_l})-\widetilde q_{m_l} }{m_l}-v_2\right\|<\frac{\delta}{6}. $$

In particular, if $N_0>0$ is the maximum length of pseudo-orbits obtained in Lemma \ref{1}, we fix $n_1=n_{k_1}$ and $m_1=m_{l_1}$ both $n_1,m_1\gg N_0$, denoting $ f^{n_1}_{x_{n_1}}(p_{n_1})=p^1$ and $f^{m_1}_{y_{m_1}}(q_{m_1})=q^1$. To follow we construct an element $(z,p_{n_1})$ in $\Sigma\times \mathbb{T}^2$ for which there is an iteration of the cocycle returning to $p_{n_1}$ and $q_{m_1}$ enough times to approximate an element of the segment between $v_1$ and $v_2$. Consequently, by Lemma \ref{1} there exist integers $n_2,\, n_3,m_2, m_3<N_0$ and sequences $x^1, x^2,y^1, y^2\in\Sigma$ such that 
$f^{n_2}_{x^1}(p^1)=p_{n_1}$, $f^{n_3}_{x^2}(p_{n_1})=q_{m_1}$, $f^{m_2}_{y^1}(q^1)=q_{m_1}$ and $f^{m_3}_{y^2}(q_{m_1})=p_{n_1}$.

For $\alpha\in[0,1]\cap\mathbb{Q}$ given by $\alpha=a/b$ with $a,b$ positive integers, we construct a periodic sequence $z\in \Sigma$ by 
$$\begin{aligned}
z_{j(n_1+n_2)+i}=&x_i&\text{ if }i&=1,\ldots {n_1}\\
z_{j(n_1+n_2)+i}=&x^1_{i-n_1}&\text{ if }i&=1+{n_1},\ldots, {n_1}+n_2,
\end{aligned} $$
for $j=0,\ldots, m_1a-1$ iterates; it continues with
$$\begin{aligned}
z_{am_1(n_1+n_2)+i}=&x^2_{i}&\text{ if }i&=1,\ldots n_3;\\
\end{aligned} $$
and we compute for $j=0,\ldots , n_1(b-a)-1$ the following elements by  
$$\begin{aligned}
	z_{am_1(n_1+n_2)+n_3+j(m_1+m_2)+i}=&y_i&\text{ if }i&=1,\ldots m_1\\
	z_{am_1(n_1+n_2)+n_3+j(m_1+m_2)+i}=&y^1_{i-m_1}&\text{ if }i&=1+m_1,\ldots, m_1+m_2;
\end{aligned} $$
finally
$$\begin{aligned}
	z_{am_1(n_1+n_2)+n_3+(b-a)n_1(m_1+m_2)+i}=&y^2_{i}&\text{ if }i&=1,\ldots m_3.\\
\end{aligned} $$
The element $(z,p_{n_1})$ is periodic for the cocycle $F$ with period $am_1(n_1+n_2)+n_3+(b-a)n_1(m_1+m_2)+m_3=bm_1n_1+N$, where $N$ has a lineal growth with respect to the variables $n_1$ and $m_1$.

Thus the integer vector $\rho(z,p_{n_1})$ can be written by 
$$\begin{aligned}
	\frac{\widetilde{f}^{bm_1n_1+N}_{z}(\widetilde p_{n_1})-\widetilde p_{n_1}}{bm_1n_1+N}&=\,\, \frac{(b-a)n_1}{bm_1n_1+N}\left(\widetilde f^{m_2}_{y^1}(\widetilde q^1)-\widetilde q^1+\widetilde f^{m_1}_{y_{m_1}}(\widetilde q_{m_1})-\widetilde q_{m_1}\right)\\
	&+\frac{am_1}{bm_1n_1+N}\left(\widetilde f^{n_2}_{x^1}(\widetilde p^1)-\widetilde p^1+\widetilde f^{n_1}_{x_{n_1}}(\widetilde p_{n_1})-\widetilde p_{n_1}\right)\\ &+\frac{C}{bm_1n_1+N},\\
\end{aligned}$$
which is equal to 
$$\begin{aligned}&\frac{(b-a)n_1}{bm_1n_1+N}\left(\widetilde f^{m_2}_{y^1}(\widetilde q^1)-\widetilde q^1\right)+\frac{(b-a)n_1m_1}{bm_1n_1+N}\rho_{m_1}(y_{m_1},q_{m_1})+\\
	&+\frac{am_1}{bm_1n_1+N}\left(\widetilde f^{n_2}_{x^1}(\widetilde p^1)-\widetilde p^1\right)+\frac{an_1m_1}{bm_1n_1+N}\rho_{n_1}(x_{n_1},p_{n_1})+
	\\ &+\frac{C}{bm_1n_1+N},\\
\end{aligned}$$
where $C$ is a constant integer vector. For concluding the result we need to find a bound for $\widetilde f^{n}_{x}(\widetilde p)-\widetilde p$ for every $(x,p)$ and $n<N_0$.
 
We assert that the map $\phi_x\colon \mathbb R^2\to \mathbb R^2$ given by $\phi_x(\widetilde p)=\widetilde f_x(\widetilde p)-\widetilde p$ is uniformly bounded for every $x\in \Sigma$. This is a consequence of $g(p)-p$ being uniformly bounded in $\mathbb T^2$ and since $g$ is homotopic to the identity, its lift is also bounded in $\mathbb{R}^2$, finally because $\widetilde f_x$ is a perturbation of at most $\epsilon$ of $\widetilde g$ we have that $\phi_x$ is in fact bounded.   

If we write $$\begin{aligned} f^{n}_{x}(\widetilde p)-\widetilde p=&\sum_{i=0}^{n-1}\left( \widetilde f_{\sigma^j(x)}(\widetilde f^{j-1}_{x}(\widetilde p))-\widetilde f^{j-1}_{x}(\widetilde p)\right)\\ =& \sum_{i=0}^{n-1} \phi_x(\widetilde f^{j-1}_{x}(\widetilde p))<N_0\sup \phi.\end{aligned}$$

Finally we can increase the indices for the sub-sequences $n_1$ and $m_1$ such that 
$$\begin{aligned} \frac{am_1}{bm_1n_1+N}\left\|\widetilde f^{n_2}_{x^1}(\widetilde p^1)-\widetilde p^1\right\|&\leq \frac{am_1N_0 \sup \phi }{bm_1n_1+N}< \frac{\delta}{6},\\
\frac{(b-a)n_1}{bm_1n_1+N}\left\|\widetilde f^{m_2}_{y^1}(\widetilde q^1)-\widetilde q^1\right\|&\frac{\leq(b-a)n_1 N_0\sup \phi}{bm_1n_1+N}<\frac{\delta}{6},\end{aligned} $$
and $$\begin{aligned}\left\|\frac{an_1m_1}{bm_1n_1+N}- \frac{a}{b}\right\|<\frac{\delta}{6},\\ \left\|\frac{(b-a)n_1m_1}{bm_1n_1+N}- \frac{b-a}{b}\right\|<\frac{\delta}{6}\end{aligned},$$
that together with equation (\ref{limit})  we conclude any point in the segment that joint $v_1$ and $v_2$ can be approximated by the rotation vector of a periodic point $(z,p_{n_1})$.
\end{proof}

\subsection{$\epsilon$-pseudo-orbits rotation set, $\rho_{mz}^\circ(\widetilde g)\neq \emptyset$}

We now turn to the more general case where \( g \) is not necessarily conservative, but its Misiurewicz-Ziemian rotation set satisfies \( \rho_{mz}^\circ(\widetilde{g}) \neq \emptyset \). The cocycle \( F \) and its lift \( \widetilde{F} \) are defined as in the previous sections. Under this assumption, we establish two lemmas that describe the structure of the sets $\Theta^+$ and $\Theta^{-}$. The first Lemma is immediate.  

\begin{lemma}
	Let \( \delta > 0 \) be such that \( \|g(p_1) - g(p_2)\| < \epsilon \) whenever \( \|p_1 - p_2\| < \delta \). Then, for every \( p \in \mathbb{T}^2 \),
	\[
	\bigcup_{i > 0} g^i(B_\delta(p)) \subset \Theta^+_\epsilon(p) \quad \text{and} \quad \bigcup_{i < 0} g^i(B_\delta(p)) \subset \Theta^-_\epsilon(p).
	\]
\end{lemma}

The next result shows that, under suitable recurrence conditions, the pseudo-orbit sets become large in a topological sense.

\begin{lemma}
	If \( p \) is a recurrent and typical point for some \( g \)-invariant ergodic measure with a totally irrational rotation vector, then the sets \( \Theta^+_\epsilon(p) \) and \( \Theta^-_\epsilon(p) \) are fully essential.
\end{lemma}
\begin{proof}Note first that if \( p \) is a recurrent point, then for any \( \varepsilon' < \varepsilon \), there exists some \( n > 0 \) such that \( \|p - g^n(p)\| < \varepsilon' \). This implies that \( B_{\varepsilon - \varepsilon'}(p) \subset \Theta^+_{\epsilon}(p) \), and then \( B_{\varepsilon}(p) \subset \Theta^+_{\epsilon}(p) \). Consequently, the set \( U_{\epsilon}^{+} = \bigcup_{i \ge 0} g^i(B_{\varepsilon}(p)) \) is contained in \( \Theta^+_{\epsilon}(p) \).
	
	Since \( p \) is recurrent and has a totally irrational rotation vector, it is a fully essential point (see \cite{inventiones1}). Therefore, the forward orbit of \( B_{\varepsilon}(p) \) under \( g \) must also be fully essential, implying that \( U_{\varepsilon}^{+} \), and hence \( \Theta^+_{\epsilon}(p) \), is fully essential. A symmetric argument applies to \( \Theta^-_{\epsilon}(p) \) by considering negative iterates.
\end{proof}

Note that, by a result in \cite{LM}, if \( \rho_{mz}(\widetilde{g}) \) has nonempty interior, then for every vector in the interior of the rotation set there exists a \( g \)-invariant ergodic measure with that rotation vector. In particular, there exist ergodic measures with totally irrational rotation vectors for that case.
	
The following lemma provides a controlled approximation property. For any given point and iteration, there exists a periodic orbit of \( p \) whose starting point is nearby and whose displacement after a bounded number of steps remains close to the displacement of the original orbit. This ensures that the rotation vectors of arbitrary orbits can be uniformly approximated by those of periodic orbits.

\begin{lemma}\label{magiclemma}
Assume that \( \rho_{mz}^\circ(\widetilde{g}) \neq \emptyset \). Given \( p \in \mathbb{T}^2 \) there are constants \( N_0 \in \mathbb{N} \) and \( L > 0 \) such that for every \( (x, \widetilde{q}) \in \Sigma \times \mathbb{R}^2 \) and \( n > 0 \) there is a vector \( w \in \mathbb{Z}^2 \), a sequence \( y \in \Sigma \) and a lift $\widetilde p\in\pi^{-1}(p)$ satisfying 
\begin{itemize}
	\item \( \widetilde{f}_y^{n'}(\widetilde{p}) = \widetilde{p} + w \), with \( n \le n' \le n + N_0 \)
	\item \( \|\widetilde{p} - \widetilde{q}\| < L \) and \( \|\widetilde{f}_y^{n'}(\widetilde{p})- \widetilde{f}_x^n(\widetilde{q})\| < L \).
\end{itemize} 
\end{lemma}

\begin{proof}


We begin by using the previous lemma to select a point \( p \in \mathbb{T}^2 \) such that both \( \Theta^{+}_{\epsilon}(p) \) and \( \Theta^{-}_{\epsilon}(p) \) are fully essential sets. It is important to observe that the future election of $N_0$ and $L$ depend only on $p$.

Since \( \Theta^{+}_{\epsilon}(p) \) is fully essential, there exists a dense open disk \( R_1 \) whose boundary consists of two essential, non-homotopically trivial curves contained in \( \Theta^{+}_{\epsilon}(p) \). Similarly, there exists a fundamental domain \( R_2 \) whose boundary is composed of essential curves lying in \( \Theta^{-}_{\epsilon}(p) \). As \( \partial R_1 \) is compact and \( \Theta^{+}_{\epsilon}(p) = \bigcup_{k} \Theta^{+,k}_{\epsilon}(p) \), there exists an integer \( N_2 > 0 \) such that \( \partial R_1 \subset \bigcup_{k=1}^{N_2} \Theta^{+,k}_{\epsilon}(p) \). Likewise, there exists \( N_1 > N_2 \) such that \( \partial R_2 \subset \bigcup_{k=1}^{N_1} \Theta^{-,k}_{\epsilon}(p) \). Let \( L_1 = \max\{ \text{diam}(R_1), \text{diam}(R_2) \} \).

For a given $n>0$ and \( (x, \widetilde{q}) \in \Sigma \times \mathbb{R}^2 \) we need to find \( \widetilde{q}' \in \mathbb{R}^2 \) satisfying $$ \| \widetilde{q} - \widetilde{q}' \| < L_1 \text{ and }\|\widetilde{f}_x^n(\widetilde{q}) - \widetilde{f}_x^n(\widetilde{q}') \| < L_1 $$with \( \pi(\widetilde{q}') \in \Theta^{+,N_1}_{\epsilon}(p) \) and \( \pi(\widetilde{f}_x^n(\widetilde{q}')) \in \Theta^{-,N_1}_{\epsilon}(p) \). Thus,  let \( \widetilde{R}_1 \) be a lift of \( R_1 \) such that \( \text{Cl}(\widetilde{R}_1) \) contains \( \widetilde{q} \), and let \( \widetilde{R}_2 \) be a lift of \( R_2 \) such that \( \text{Cl}(\widetilde{R}_2) \) contains \( \widetilde{f}_x^n(\widetilde{q}) \). Since both \( \text{Cl}(\widetilde{R}_1) \) and \( \text{Cl}(\widetilde{R}_2) \) are fundamental domains and \( \widetilde{f}_x^n(\text{Cl}(\widetilde{R}_1)) \) intersects \( \text{Cl}(\widetilde{R}_2) \), there exists a point \( \widetilde{q}' \in \partial \widetilde{R}_1 \) such that \( \widetilde{f}_x^n(\widetilde{q}') \in \partial \widetilde{R}_2 \). Clearly \( \widetilde{q}' \) satisfies all the required properties.

Now, since \( q' = \pi(\widetilde{q}') \in \partial R_1 \subset \Theta^{+,N_1}_{\epsilon}(p) \), there exists a word \( z^1 \in \Sigma \) and an integer \( 0 \le m_1 \le N_1 \) such that \( f_{z^1}^{m_1}(p) = q' \), here we choose the lift $\widetilde{p}$ such that $\widetilde f_{z^1}^{m_1}(\widetilde p) =\widetilde  q' $. Similarly, since \( \pi(\widetilde{f}_x^n(\widetilde{q}')) \in \partial R_2 \subset \Theta^{-,N_1}_{\epsilon}(p) \), there exists a word \( z^2 \in \Sigma \) and an integer \( 0 \le m_2 \le N_1 \) such that \( f_{z^2}^{m_2}(f_x^n(q')) = p \). 

We fix as \( L_2 > 0 \) the constant satisfying that for all \( (z, \widetilde{r}) \in \Sigma \times \mathbb{R}^2 \) and every \( 0 \le i \le N_1 \), we have \( \| \widetilde{f}_z^i(\widetilde{r}) - \widetilde{r} \| < L_2 \). Let \( L = L_1 + L_2 \), \( N_0 = 2N_1 \), and \( n' = n + m_1 + m_2 \). Define the word \( y \in \Sigma \) by
\[
y_i = 
\begin{cases}
	z^1_i & \text{for } i = 1, \ldots, m_1, \\
	x_{i - m_1 - 1} & \text{for } i = m_1 + 1, \ldots, m_1 + n, \\
	z^2_{i - m_1 - n - 1} & \text{for } i > m_1 + n.
\end{cases}
\]
Taking \( w = \widetilde{f}_y^{n'}(\widetilde{p}) - \widetilde{p} \) and considering that
$$\begin{aligned}\|\widetilde f^{n'}_y(\widetilde p)-\widetilde f^n_x(\widetilde q)\|=&\|\widetilde f^{m_2}_{z^2}\widetilde f_x^{n}\widetilde f_{z^1}^{m_1}(\widetilde p)-\widetilde f^n_x(\widetilde q)\|\\
=& \|\widetilde f^{m_2}_{z^2}\widetilde f_x^{n}(\widetilde q')-\widetilde f^n_x(\widetilde q)\|\\
\leq &  \|\widetilde f^{m_2}_{z^2}(\widetilde f_x^{n}(\widetilde q'))-\widetilde f_x^{n}(\widetilde q')\|+\|\widetilde f_x^{n}(\widetilde q')-\widetilde f^n_x(\widetilde q)\| \\
\leq &L_2+L_1=L	\end{aligned}$$
we conclude the proof. 
\end{proof}

\begin{prop}
	If $\rho_{mz}^\circ(g)\neq \emptyset$, then $\rho_{mz}(\widetilde F)$ is convex. 
\end{prop}	

\begin{proof}
As in the previous lemma, we select a point \( p_0 \in \mathbb{T}^2 \) such that both \( \Theta^{+}_{\epsilon}(p_0) \) and \( \Theta^{-}_{\epsilon}(p_0) \) are fully essential sets.

Our goal is to prove that given two points \( v_1, v_2 \in \rho_{mz}(\widetilde F) \), any convex combination of them is contained in \( \rho_{mz}(\widetilde F) \). We will show, for \( \alpha \in [0,1] \cap \mathbb{Q} \), that \( \alpha v_1 + (1-\alpha)v_2 \in \rho_{mz}(\widetilde F) \), which is sufficient since \(\rho_{mz}\) is closed.

For each $i\in\{1,2\}$, let \((x^{i,k}, p^i_k)_k\) be a sequence such that
\[
\lim_{k\to\infty} \frac{\widetilde{f}^{n_{i,k}}_{x^{i,k}}(\widetilde{p}^i_k) - \widetilde{p}^i_k}{n_{i,k}} = v_i.
\]
 Lemma~\ref{magiclemma} implies that for the sequences \((x^{i,k}, p^i_k)\) and iterates \(n_{i,k}\), there exist sequences \(y^{i,k}\), integers \(m_{i,k}\) with \(n_{i,k} \le m_{i,k} \le n_{i,k} + N_1\), and vectors \(w_{i,k} \in \mathbb{Z}^2\) such that:
\begin{itemize}
	\item \(\widetilde{f}^{m_{i,k}}_{y^{i,k}}(\widetilde{p}_0) - \widetilde{p}_0 = w_{i,k}\),
	\item \(\| w_{i,k} - (\widetilde{f}^{n_{i,k}}_{x^{i,k}}(\widetilde{p}^i_k) - \widetilde{p}^i_k) \| < 2L\).
\end{itemize}
Since
$$\begin{aligned}
	\left\| \frac{w_{i,k}}{m_{i,k}} -v_i\right\|&\leq \left\|\frac{w_{i,k}-(\widetilde{f}^{n_{i,k}}_{x^{i,k}}(\widetilde{p}^i_k)-\widetilde{p}^i_k)}{m_{i,k}}\right\|+\left\|\frac{\widetilde{f}^{n_{i,k}}_{x^{i,k}}(\widetilde{p}^i_k) - \widetilde{p}^i_k}{m_{i,k}}-v_i\right\|\\
	&\leq \left\| \frac{2L}{m_{i,k}}\right\|+\left\|\left(\frac{n_{i,k}}{m_{i,k}}\right)\frac{\widetilde{f}^{n_{i,k}}_{x^{i,k}}(\widetilde{p}^i_k) - \widetilde{p}^i_k}{n_{i,k}}-v_i\right\|,
\end{aligned}
$$
and $n_{i,k}/m_{i,k}$ approaches to $1$ as $k$ increases, the limit of ${w_{i,k}}/{m_{i,k}}$ is $v_i$ for each $i$. Fixing \(\delta > 0\) and taking \(k\) sufficiently large such that both \(\| \frac{w_{i,k}}{m_{i,k}} - v_i \| < \delta\), we rename 
$y^{i,k} = y^i, \  m_{i,k} = M_i, \  w_{i,k} = w_i.$

Let \(\alpha = \frac{a_1}{(a_1 + a_2) }\) for positive integers \(a_1, a_2\), we see inductively that if we take \(z\) to be a periodic word of period \((a_1 + a_2) M_1 M_2\) defined as
\[
z_i = 
\begin{cases}
	y^{1}_{i - [\frac{i}{M_1}] M_1} & \text{for } 0 \leq i \le a_1 M_1 M_2, \\
	y^2_{i - a_1 M_1 M_2 - [\frac{i}{M_2}]M_2} & \text{for } a_1 M_1 M_2+1 \le i < (a_1+a_2) M_1 M_2,
\end{cases}
\]
then
\begin{align*}
	\widetilde{f}^{j M_1}_z(\widetilde{p}_0) - \widetilde{p}_0 &= j w_1 \quad \text{for } 0 \le j \le a_1 M_2, \\
	\widetilde{f}^{a_1 M_1 M_2 + j M_2}_z(\widetilde{p}_0) - \widetilde{p}_0 &= a_1 M_2 w_1 + j w_2 \quad \text{for } 0 \le j \le a_2 M_1.
\end{align*}
Therefore, \((z, \widetilde p_0)\) is a periodic point of period $(a_1 + a_2) M_1 M_2$ for \(F\), and its rotation vector is
\[
\frac{a_1 M_2 w_1 + a_2 M_1 w_2}{(a_1 + a_2) M_1 M_2} = \alpha \frac{w_1}{M_1} + (1 - \alpha) \frac{w_2}{M_2},
\]
which is \(\delta\)-close to \(\alpha v_1 + (1 - \alpha) v_2\).
\end{proof}

\end{document}